\documentclass[a4paper,preprint,12pt]{elsarticle} 
\usepackage[english]{babel} 
\usepackage[latin1]{inputenc}
\usepackage{amssymb}  
\usepackage{amsmath}    
\usepackage{pb-diagram}         
\usepackage{epstopdf}                   
\usepackage{graphicx}          
 \usepackage{epsfig}      
\usepackage{color} 
\usepackage{fancyhdr}   
 \usepackage{natbib}  
\usepackage{rotating}   
\usepackage[T1]{fontenc}        
\usepackage{url}                
\usepackage{mathtools}                 
\usepackage{vmargin}    
\usepackage{lineno}
\usepackage{setspace}
\usepackage{color,soul}
\usepackage[colorlinks=true,linkcolor= blue,urlcolor=blue,citecolor=blue]{hyperref} 
\usepackage[ruled,vlined]{algorithm2e}  
\usepackage[title]{appendix} 
 \setulcolor{blue} 
 \setstcolor{blue}   
 \sethlcolor{red}  
\biboptions{square}  
     
\newcommand{\bo}[1]{\mathbf{#1}} 

\def\indic{\hbox{1\kern-.24em\hbox{I}}}      

\newcommand{\cov}{\mathbb{C}_{ov}}   
\newcommand{\var}{\mathbb{V}}  
\newcommand{\esp}{\mathbb{E}}    
\newcommand{\trace}{\text{Tr}}

\newcommand{\x}{x}
\newcommand{\X}{X}

\newcommand{\Z}{Z} 
\newcommand{\R}{\mathbb{R}}

\newcommand{\MG}{G} 
 
\newcommand{\XE}{\mathbb{\underline{X}}}

\newcommand{\M}{f}

\newcommand{\T}{T}    
      
\newcommand{\NN}{n}     
\newcommand{\Select}{\mathbf{s}}  

\newcommand{\norme}[1]{\left|\left| #1 \right|\right|_{L^2}}

{\bf}{\it} 
{\bf}{\it}  

\newtheorem{theorem}{Theorem}{\bf}{\it}     
\newtheorem{lemma}{Lemma}{\bf}{\it}          
{\bf}{\it} 
\newtheorem{corollary}{Corollary}{\bf}{\it} 
\sethlcolor{green}        
  
\catcode`\"=\active 
\catcode`\"=\active 
\def"{\og\ignorespaces}
\def"{{\fg}} 
 
\makeatletter
\def\ps@pprintTitle{%
  \let\@oddhead\@empty
  \let\@evenhead\@empty
  \def\@oddfoot{\reset@font\hfil\thepage\hfil}
  \let\@evenfoot\@oddfoot
}
\makeatother     
   
\begin{document}

\begin{frontmatter}         
\title{Sharp inequalities between variance-based dependent sensitivity indices and Shapley effects: upper-bounds} 
\author[a,b]{Matieyendou Lamboni\footnote{Corresponding author: matieyendou.lamboni[at]univ-guyane.fr/gmail.com,  May 27, 2026}}         
\address[a]{University of Guyane, Department DFR-ST, French Guiana, France} 
\address[b]{228-UMR Espace-Dev, University of Guyane, University of R\'eunion, IRD, University of Montpellier, France.}              
                                              
\begin{abstract}   
For models evaluated at a random set of independent variables, the variance-based Shapley effects range between  Sobol' indices, and the corresponding total indices admit derivative-based upper-bounds. Such  relationships fail when the inputs are non-independent. This study investigates a general inequality link between the variance-based dependent sensitivity indices, recently introduced by us, and the variance-based Shapley  effects for models with dependent inpus. It turns out that Shapley effects range  between the main and total dependent sensitivity indices, and such indices are computationally more attractive. Moreover, different upper-bounds of such indices are provided so as to ease the identification of non-relevant inputs in higher dimensions as well as to obtain  sometimes practical estimates of total dependent indices. Some of such bounds rely on the traditional gradients, while others rely on generalized sensitivity indices using dependency models.  
\end{abstract}

               
\end{frontmatter}         
\setpagewiselinenumbers            
\modulolinenumbers[1]   

  
\section{Introduction}  
Traditional ANOVA decompositions (\cite{hoeffding48a,scheffe59,efron81,antoniadis84,sobol93,rabitz99,kuo10,gelman05}) and derivative-based ANOVA (\cite{lamboni22,lamboni25stats}) are interesting frameworks in  exploratory and confirmatory models analysis. The corresponding decompositions are unique when all the input variables are independent, leading to unique definitions of variance-based sensitivity indices that enjoy nice properties (\cite{saltelli00}). To list some of them, consider a real-valued model $\M : \R^d \to \R$ with $d >0$, which is evaluated at $d$ independent variables, that is, $\bo{\X} := (X_j, \, \forall\, j \in \{1, \ldots, d\})$.  
For every input $\X_j$, the main Sobol' index (i.e., $S_j$), the total index ($S_{T_j}$) and the variance-based (Vb) Shapley effect of  $\X_j$  (i.e., $Sh_j$ from \cite{owen14b,owen17,iooss19}) verify the relationship     
$
0\leq S_j \leq Sh_j \leq S_{T_j} \leq 1                                                        
$,      
showing how Sobol' indcies bracket the Vb-Shapley effects (\cite{owen17}).\\ 
 
Such inequalities are supplemented by different upper-bounds of the total indices thanks to Poincar\'e-type inequalities. While different upper-bounds have been provided in \cite{kucherenko09,roustant14,kucherenko16,roustant17,lamboni22}, the sharp upper-bound (i.e., $ U\!B_j$) of the total index of $\X_j$ that holds universally across all of the distribution functions is provided in \cite{lamboni22} (Theorem 3), leading to 
$  
0\leq S_j \leq Sh_j \leq S_{T_j} \leq U\!B_j  
$. The equality $S_{T_j} = U\!B_j$ holds for specific functions constructed in \cite{lamboni22}. Since (i) $S_j,\, U\!B_j$ are computationally more attractive than $S_{T_j}$ (\cite{kucherenko09,lamboni22,ghanem17}), and (ii) $S_{T_j}$ is computationally more attractive than $Sh_{j}$ (\cite{owen14b}), the above inequalities enable sometimes good approximations of the values of $S_{T_j}$s and $Sh_j$s using fewer model runs. It is to be noted that randomized computations of $Sh_j$s proposed in \cite{goda21} require the same model runs compared to direct computations of $S_{T_j}$s.
 Generally, $U\!B_j$s serve as a screening measure related to Sobol' indoces. Therefore, establishing such types of inequalities for models with non-independent inputs is worth investigating and is the key motivation of this paper.\\    
        
For functions evaluated at non-independent input variables $\bo{\X}$, different variance-based global sensitivity indices have been proposed in (i) \cite{hooker07,kucherenko12,rabitz12,chastaing15,kucherenko17} using $L_2$ projections and the orthogonalization procedure; (ii) \cite{mara12,mara15,tarantola17,lamboni21,lamboni25math} using probabilistic transformations of dependent or correlated inputs, such as   
 the inverse Rosenblatt (\cite{obrien75}) and Nataf transforms or dependency models (\cite{skorohod76,lamboni21,lamboni22}). Note that dependency models (DMs) allow for extracting the dependency structures of non-independent variables under the statistical and probabilistic framework.   
While the first-type aforementioned approaches can sometimes gives inconsistent effects of the input variables, one always obtains consistent main and total Vb-sensitivity indices using the second-type approaches. Basically, the latter approaches consist in applying Sobol' indices on equivalent representations (ERs) of the model outputs, which includes only independent inputs. Consequently, the extended Sobol' indices of the input $\X_j$ (i.e., $eS_j$ and $eS_{T_j}$) verify the inequalities $0\leq eS_j \leq eS_{T_j} \leq 1$. Moreover, ERs resulting directly from dependency models enable the formulation of the derivative-based upper-bounds of $eS_{T_j}$s, ensuring    
$$     
0 < eS_j \leq eS_{T_j} \leq  eU\!B_j \, ,                  
$$  
with $eU\!B_j$ being the shparp upper-bound of $eS_{T_j}$ (\cite{lamboni26mcs}), likewise $U\!B_j$ is that of $S_{T_j}$. \\  
                      
Establishing a general link between the extended Sobol' indices and Shapley effects in the presence of non-independent variables has been investigated in \cite{iooss19}. According to that work, there is no universal link between both types of inputs' effects, as $Sh_j$ can be less or greater than $eS_j$ and $eS_{T_j}$ across different functions. Therefore, a formal inequality link between Shapley effects and  extended Sobol' indices (\cite{mara15,tarantola17,lamboni21,lamboni25stats}) might be difficult even impossible to obtain, because it can happen that the sum of $eS_j$s is greater than one. \\
    
The Vb-dependent sensitivity indices (i.e., DSIs), recently developed in \cite{lamboni26uq}, consist in applying probabilistic transforms across all possible conditional inputs of the form $\X_j |\bo{\X}_u$ for all $u \subseteq \{1, \ldots, d\}$, with the resulting sensitivity measure averaging over the effects of such conditional inputs in a symmetric way. Precisely, the resulting transforms a.k.a. dependency models of $\bo{X}$ and the corresponding ERs of $\M(\bo{X})$ induce a new multivariate model, and DSIs follow directly by applying the first-type generalized sensitivity indices (\cite{lamboni11,gamboa14}). Such indices 
are promising for establishing a formal inequality link with Shapley effects, as the corresponding main indices and interactions sum to one. Moreover, DSIs can also be expressed in ways that look like Shapley effects, but they differ from the Shapley effects, except for linear models evaluated at the Gaussian random vectors having a full rank covariance matrix.  \\          
                          
In this paper, one investigates a sharp inequality link between DSIs provided in \cite{lamboni26uq} and the Vb-Shapley effects for single-response models evaluated at non-independent variables. Additionally, one provides upper-bounds of total DSIs as well as the Vb-Shapley effects using the traditional gradients of models or generalized sensitivity indices of dependency models. The rest of the paper is organized as follows: Section \ref{sec:pre} provides dependency models of non-independent variables and expressions of DSIs. Such expressions are well-suited for (i) establishing the aforementioned links (see Sections \ref{sec:link1}-\ref{sec:link2}), and (ii) deriving different upper-bounds of the total DSIs in Section \ref{sec:up}. Illustrations of the proposed approach are provided in Section \ref{sec:test}, and Section \ref{sec:con} concludes this work.

\section*{General notation}                         
Given $u \subseteq \{1,\ldots, d\}$ with the cardinality $|u|>0$, we use $\bo{\X}_u :=(\X_j,\, \forall\, j \in u)$ and $\bo{\X}_{\sim u} :=(\X_j,\, \forall\,  j\in \{1,\ldots, d\}\setminus u)$ for two subsets of $\bo{\X}$. We also use $Y_1 \stackrel{d}{=} Y_2$ to say that $Y_1$ and $Y_2$ have the same distribution function (DF). 
The input $\X_j$ has $F_j$ as the DF ($\X_j \sim F_j$), and continuous inputs $\X_k$s have $\rho_k$s as probability density functions.   

A set of $d$ non-independent variables $\bo{\X}$ can always be organized as a tuple of $K$ sets of random variables $\bo{\X}_{\pi_1}, \ldots, \bo{\X}_{\pi_K} $, which are mutually independent, where the sets $\pi_1, \ldots,\pi_K$ form a partition of $\{1, \ldots, d\}$. Without loss of generality, we use $\bo{\X}_{\pi_1}$ for a random set of $d_1\geq 0$ independent variable(s); $\bo{\X}_{\pi_k}$ for one set of $d_k\geq 2$ dependent variables with $k =2, \ldots, K$. Evaluating $\M$ at $\bo{\X}$ yields $\M(\bo{\X})$ as the random output. \\  
Given a matrix $\mathcal{A} \in \R^{\NN\times \NN}$, $\trace(\mathcal{A})$ stands for the trace of $\mathcal{A}$. Also, $\esp[\cdot]$ (resp. $\var[\cdot]$) stands for the expectation (resp. variance-covariance) operator.

\section{Generic formulation of dependent sensitivity indices} \label{sec:pre} 
\subsection{Equivalent representations of outputs and symmetric models} \label{sec:sym} 
Let $\bo{w}_k := (w_{1,k}, \ldots, w_{d_k, k})$ be an arbitrary permutation of $\pi_k$ with $k=2, \ldots K$.  
For a given $w_{1,k}$, we use $\bo{\X}_{\sim w_{1,k}} :=\left(\X_{w_2, k}, \ldots, \X_{w_{d_k}, k} \right)$ for a $d_k-1$-dimensional random vector. The permutation $\bo{w}_k$ leads to a DM of $\bo{\X}_{\pi_k}$ (\cite{skorohod76,lamboni21,lamboni22b,lamboni24sank,lamboni25math}), that is,                
\begin{equation} \label{eq:mdepk}           
              \bo{\X}_{\sim w_{1,k}} \stackrel{d}{=} r_{w_{1,k}}\left(\X_{w_{1,k}}, \bo{Z}_{\sim w_{1,k}} \right), \qquad   k =2, \ldots, K \,   ;            
\end{equation}        
where $\bo{Z}_{\sim w_{1,k}} :=\left(Z_{w_2, k}, \ldots, Z_{w_{d_k}, k} \right)$ is a vector of independent variables, and it is independent of $\X_{w_{1,k}}$ as well. The innovation variables $\bo{Z}_{\sim w_{1,k}}$ contain the remaining information about $\bo{\X}_{\sim w_{1,k}}$ once $\X_{w_{1,k}}$ is known. Composing $\M(\bo{\X})$ by (\ref{eq:mdepk}) yields       
\begin{equation} \label{eq:indeg}    
 \M\left(\bo{\X}_{\pi_1}, \X_{w_{1,2}}, 
r_{w_{1,2}}\left(\X_{w_{1,2}}, \bo{Z}_{\sim w_{1,2}} \right), \ldots, 
\X_{w_{1, K}}, r_{w_{1,K}}\left(\X_{w_{1,K}}, \bo{Z}_{\sim w_{1,K}} \right)  \right) =: g\left(\bo{\X}_{\pi_1},  \bo{\X}_{\Select}, \bo{Z} \right) \, ,     
\end{equation}        
with  $\Select := \{w_{1,2}, \ldots, w_{1,K} \}$        
$
\bo{\X}_{\Select} := \left(\X_{\ell},  \forall\, \ell \in \Select \right)
$ and 
$
\bo{Z} := \left(\bo{Z}_{\sim w_{1,k}}, \forall\, k \in \{2, \ldots, K\} \right) 
$.    
The function $g$ includes only independent variables, and we are able to recover the distributions of the output $\M(\bo{\X})$ conditional on subsets of inputs (i.e., $\M(\bo{\X}) | \bo{\X}_{u}$)  using $g$ and its inputs (see \cite{lamboni25math}). Moreover, the meaning of innovation variables w.r.t. the original output $\M (\bo{\X})$ has been examined in \cite{lamboni26uq}. To recall such a result in Lemma \ref{lem:condeff}, consider $j_k \in \pi_k$ with $k \in \{2, \ldots, K\}$, and note that $Z_{j_k}$ is the innovation variable that represents $\X_{j_k}$ in a DM when $\X_{j_k}$ is not at the first position. Given a permutation $\bo{w}_{k}$ of $\pi_k$, denote with $p_{j_k}$ the position of 
$j_k$ in that permutation, that is, $w_{p_{j_k}, k} =j_k$ and 
$\bo{w}_{k} = (w_{1,k}, \ldots, w_{p_{j_k}-1,k}, j_k, \ldots, w_{d_k,k})$. Using $u_k=\{w_{1,k}, \ldots, w_{p_{j_k}-1,k}\}$ and $\bo{w}_{\sim p_{j_k}} := (w_{p_{j_k}+1,k}, \ldots, w_{d_k,k})$ leads to $u_k \subseteq \pi_{k} \setminus \{j_k\}$.                         
\begin{lemma} (\cite{lamboni26uq}) \label{lem:condeff}                
Let $\bo{\X}_{\sim \pi_{k}} := \left(\bo{\X}_{\pi_1}, \bo{\X}_{\pi_2},\ldots, \bo{\X}_{\pi_{k-1}}, \bo{\X}_{\pi_{k+1}}, \ldots, \bo{\X}_{\pi_{K}}  \right)$ and  consider the ER 
$g_{u_k}\left( \bo{\X}_{u_{k}},\, Z_{j_k},\, \bo{Z}_{\bo{w}_{\sim p_{j_k}}}, \, \bo{\X}_{\sim \pi_{k}} \right) :=\M\left(\bo{\X}_{u_{k}}, r_{u_k}\left(\bo{\X}_{u_{k}},\, Z_{j_k}, \bo{Z}_{\bo{w}_{\sim p_{j_k}}}\right), \, \bo{\X}_{\sim \pi_{k}}  \right)$. Then,                  
\begin{equation}        
  \M(\bo{\X}) \, | \left( \X_{j_k} | \bo{\X}_{u_{k}}=\bo{x}_{u_{k}} \right) \stackrel{d}{=} g_{u_k}\left( \bo{x}_{u_{k}},\, Z_{j_k},\, \bo{Z}_{\bo{w}_{\sim p_{j_k}}}, \, \bo{\X}_{\sim \pi_{k}} \right) \, | Z_{j_k} \, .  \nonumber       
\end{equation}      
\end{lemma}        
     
It turns out that the Vb-main-effect of the innovation variable $Z_{j_k}$ associated with $g_{u_k}(\cdot)$ is equal to the effect of the conditional input $\X_{j_k} | \bo{\X}_{u_{k}}= \bo{x}_{u_{k}}$ using $\M(\bo{\X})$. Such effects may change according to the position $p_{j_k}$ and the set $u_k$.  
For a given position $p_{j_k} \in \{1, \ldots, d_k\}$ of $\X_{j_k}$, we have $\binom{d_k-1}{p_{j_k}-1}$ different effects of  $Z_{j_k}$, corresponding to the selection of the set $u_k$ out of $d_k-1$ elements.  Among the $d_k$ possible positions of $\X_{j_k}$ in DMs, the position given by $\left[\frac{d_k -1}{2} \right] +1$ is going to bring the maximum number of different effects of $Z_{j_k}$, with $[a]$ being the largest integer that is less than the real $a$. To be able to assess each different effect of every input of $\bo{\X}_{\pi_k}$ we need at least   
$  
R_{0,k} :=d_k \binom{d_k-1}{\left[\frac{d_k -1}{2} \right]}                   
$ 
 ERs because two different inputs cannot simultaneously be at the same position in a given DM. Different effects of every input of $\bo{\X}_{\pi_k}$ related to the position $p_k \in \{1, \ldots, d_k\}$ are repeated
$  
 d_k^{-1} R_{0,k} \, \binom{d_k-1}{p_k-1}^{-1}    
$
times (\cite{lamboni26uq}).\\              
For all of the $d$ inputs, the unique, minimal and fairness or symmetric model that accounts for the aforementioned different effects of every input requires the number of ERs  given by
 (\cite{lamboni26uq}): 
$$
R_{\min}^s  := lcm \left(d_k \binom{d_k-1}{p_k-1}, \quad  p_k=1, \ldots, d_k, \,  k=2, \ldots, K  \right) \, ,  
$$       
with $lcm$ being the least common multiple function. Thus, the number of repeated effects for each position $p_k \in \{1, \ldots, d_k\}$ of an input is given by           
\begin{equation}     \label{eq:replam}
\lambda_{p_k,k} := d_k^{-1} R_{\min}^s \, \binom{d_k-1}{p_k-1}^{-1}; \quad k=2, \ldots, K \, .   
\end{equation}                 
Denote with $\MG(\cdot)$ the symmetric multivariate model having  $R_{\min}^s$ ERs as outputs.

\subsection{Dependent sensitivity indices}\label{sec:DMSA}
Applying the first-type generalized sensitivity (\cite{lamboni21,gamboa14}) to the model $\MG(\cdot)$ yields the DSIs of non-independent variables. Consequently, the main DSIs and interactions sum to one (see \cite{lamboni26uq}). 
 Keeping in mind Lemma \ref{lem:condeff} and knowing that the first-type GSIs are defined via the trace function, DSIs are reformulated below by using $\Sigma$ for the variance of $\M(\bo{\X})$.\\   

Firstly, note that each $\X_{j_1}$ of $\bo{\X}_{\pi_1}$ (i.e., $j_1  \in \pi_1$) has only one effect, which is repeated $\lambda_{1,1} =R_{\min}^s$ times.  Thus, the main and total DSIs of $\X_{j_1}$ are    
\begin{equation}  
DS_{j_1}  := \frac{\var\left\{ \M_{j_1}^{fo}(\X_{j_1}) \right\}}{\Sigma}\, ;
\qquad \quad   
DS_{T_{j_1}} := \frac{\var\left\{ \M_{j_1}^{tot}(\bo{\X}) \right\}}{\Sigma} \, ,       
\end{equation}     
where   
$
\M_{j_1}^{fo}(\X_{j_1}) := \esp\left[\M\left(\bo{\X}\right) | \X_{j_1} \right] - \esp\left[\M\left(\bo{\X}\right) \right]
$
and 
$
\M_{j_1}^{tot}(\bo{\X}) := \M\left(\bo{\X}\right) - \esp_{\X_{j_1}}\left[\M\left(\bo{\X}\right) \right] 
$ 
are rhe first-order and total sensitivity functionals (SFs) of $\X_ {j_1}$, and $\esp_{\X_{j_1}}$ stands for the expectation taken w.r.t. $\X_{j_1}$. Obviously, we have the inequality (\cite{lamboni22}) 
$$
DS_{j_1} \leq DS_{T_{j_1}} \leq  
  D\!U\!B_{j_1} := \frac{1}{2 \var\left[\M(\bo{\X}) \right]} \esp\left[\left(\frac{\partial\M}{\partial x_{j_1}} (\bo{\X}) \right)^2 \frac{ F_{j_1}(\X_{j_1}) \left(1- F_{j_1}(\X_{j_1})\right)}{\left[\rho_{j_1}(\X_{j_1})\right]^2} \right] \, .    
$$     
         
Secondly, for $k> 1$ and for any $j_k \in \pi_k$, the  permutation $\bo{w}_{k} = (u_k, w_{p_{j_k}}=j_k, \bo{w}_{\sim p_{j_k}})$ leads to the following ER:   
$$  
g_{u_k,j_k}\left(\X_{w_{1,k}}, \bo{Z}_{\sim w_{1,k}} , \bo{\X}_{\sim \pi_k} \right) := 
\M\left(\X_{w_{1,k}}, r_{_{w_{1, k}}} \left(\X_{w_{1,k}}, \bo{Z}_{\sim w_{1, k}}  \right), \, \bo{\X}_{\sim \pi_k}  \right) \, ,    
$$   
with the innovation variable  $Z_{j_k}$ being part of $\bo{Z}_{\sim w_{1, k}}$ when $p_{j_k} \neq 1$ and  $Z_{j_k} = \X_{j_k}$ when $p_{j_k}=1$. The  SFs of $\X_{j_k}$ are given as follows:   
$$
\M_{j_k|u_k}^{fo}(Z_{j_k}) := \esp\left[g_{u_k,j_k}\left(\X_{w_{1,k}}, \bo{Z}_{\sim w_{1,k}} , \bo{\X}_{\sim \pi_k} \right) | Z_{j_k} \right]  - \esp\left[g_{u_k,j_k}\left(\X_{w_{1,k}}, \bo{Z}_{\sim w_{1,k}} , \bo{\X}_{\sim \pi_k} \right) \right] \, , 
$$ 
$$ 
\M_{j_k|u_k}^{tot}\left(\X_{w_{1,k}}, \bo{Z}_{\sim w_{1,k}} , \bo{\X}_{\sim \pi_k} \right) := g_{u_k,j_k}\left(\X_{w_{1,k}}, \bo{Z}_{\sim w_{1,k}} , \bo{\X}_{\sim \pi_k} \right) - \esp_{Z_{j_k}}\left[g\left(\X_{w_{1,k}}, \bo{Z}_{\sim w_{1,k}} , \bo{\X}_{\sim \pi_k} \right) \right] \, . 
$$    
Taking the variance as the importance measure gives the main and total DSIs of $\X_{j_k}$, that is, \cite{lamboni26uq} 
\begin{equation}  
DS_{j_k} =  \frac{1}{d_k \Sigma} \sum_{\substack{u_k \subseteq  \pi_{k} \setminus \{j_k\}}}
 \binom{d_k-1}{|u_k|}^{-1}  \sigma_{u_k, j_k}^{fo} \, ,       
\end{equation}      
\begin{equation}  \label{eq;todsinew}
DS_{T_{j_k}} = \frac{1}{d_k \Sigma} \sum_{\substack{u_k \subseteq  \pi_{k} \setminus \{j_k\}}}
 \binom{d_k-1}{|u_k|}^{-1}  \sigma_{u_k, j_k}^{tot} \, ,   
\end{equation}       
where      
$$
\sigma_{u_k, j_k}^{fo} := \var\left\{ \M_{j_k|u_k}^{fo}(Z_{j_k}) \right\} ,
\qquad \quad
\sigma_{u_k, j_k}^{tot} := \var\left\{ \M_{j_k|u_k}^{tot}\left(\X_{w_{1,k}}, \bo{Z}_{\sim w_{1,k}} , \bo{\X}_{\sim \pi_k} \right) \right\} \, .             
$$   

The definition of DSIs for a group of inputs can be found in \cite{lamboni26uq}.  Computations of all the main and total DSIs require at most $C_l := 4m d_{\max} \binom{d_{\max}-1}{\left[\frac{d_{\max}-1}{2} \right]}$ model runs with $m$ being a given sample size and $d_{\max} := \max\left\{d_2, \ldots, d_K \right\}$. Note that $C_l \leq  4m d \binom{d-1}{\left[\frac{d-1}{2} \right]} \leq  2m d \, 2^d$. For independent variables, DSIs are exactly Sobol' indices.      
		           
\subsection{Variance-based Shapley effects}
Shapley values (\cite{shapley53}) are used in the game theory for distributing a given reward among teams according to a coalition worth value function. The variance-based Shapley effects make use of the first-order-effect variance as a coalition worth value, leading to a global indicator of inputs' importance, especially when the inputs are non-independent (\cite{owen14b,song16,owen17,iooss19}). Note that such effects mix out the main and interaction effects in only one index per input. In the same sense, derivative-based Shapley values have been considered in \cite{duan24,lamboni26mcs} for assessing the effects of inputs using model derivatives, including dependent gradients.\\   
Formally, the Vb-Shapley effects of the input $\X_{j}$ for any $j \in\{1, \ldots, d\}$ is defined by (\cite{owen14b,owen17})   
$$
Sh_{j} =  \frac{1}{d \Sigma} \sum_{\substack{u \subseteq  \{1, \ldots, d \} \setminus \{j\}}}
\binom{d-1}{|u|}^{-1}  
\left\{ \var\left\{ \esp\left[\M\left(\bo{\X}\right) | \bo{\X}_{u\cup \{j\}} \right] \right\} 
- \var\left\{ \esp\left[\M\left(\bo{\X}\right) | \bo{\X}_{u} \right] \right\}  \right\} \, .  
$$ 
    
For linear models and linear probabilistic distributions, the following equalities between Shapley effects and DSIs hold.  
\begin{theorem} (\cite{lamboni26uq}).
Assume that $\M$ is a linear model, and $\bo{\X}$ follows the Gaussian distribution with a full rank covariance matrix. Then,  
$$
Sh_{j} = DS_j = DS_{T_j}  \, .     
$$
\end{theorem}    
  
Direct computations of all the Shapley effects require $C :=N_i N_0 d! (d-1) +N_v$  model evaluations, where $N_i$ and $N_0$ are the inner and outer loop sizes, and $N_v$ is the sample size for the variance computation (see \cite{song16,iooss19} for more details). For instance, when $\bo{\X}$ consists only of  one block of dependent variables, we can check that $C_l = 4N_v \frac{d!}{ \left[ \frac{d-1}{2}\right]! \,  \left(d-1 - \left[\frac{d-1}{2}\right] \right)!} \leq C$ for moderate and high dimensional models. Therefore, we have  $C_l \leq C$ in general, meaning that DSIs require less model evaluations compared to Shapley effects. Concerning Shapley effects, substantial reductions of the computational cost have been achieved in (i) \cite{plischke21}, leading to the total cost of  magnitude $2^d$, and in (ii) \cite{benoumechiara21} by randomly choosing $m$ permutations out of $d!$, resulting in the total cost $C' :=N_i N_0 m (d-1) +N_v$. For independent variables, it seems that the cost  $N(d+1)$ can be reached using  a randomized version of Shapley effects (\cite{goda21}).   
						      
\section{Main results}
\subsection{Links between dependent sensitivity indices and Shapley effects} 
For the sake of pedagogy concerning the proofs, we start this section with a general inequality link between Shapley effects and DSIs for one group of dependent variables, followed by inequalities involving both types of inputs' effects for a tuple of $K$ sets of random variables that are mutually independent.
   
\subsubsection{Case of one group of dependent variables} \label{sec:link1}
In this section, we assume that $\bo{\X}$ consists of dependent variables, and it cannot be spitted out into two or more sets of independent variables. By notation, $d=d_2,\, j_2=j,\;  \pi_{2} = \{1, \ldots, d\}$ in this section. \\     

For a given $j \in \pi_{2}$ and $u \subseteq \pi_{2} \setminus\{j \}$ with $|u_2|=p_j-1$, there exists a permutation $\bo{w}$ of $\pi_{2}$ such that $u = \{w_{1}, \ldots, w_{p_j-1} \}$ and $j= w_{p_j}$. By using $\boldsymbol{\varpi} := (w_{2}, \ldots, w_{p_j-1})$, the Vb-Shapley effect of $\X_{j}$ becomes   
\begin{eqnarray}  
 Sh_{j}  &=&  \frac{1}{d \Sigma} \sum_{\substack{u \subseteq  \pi_{2} \setminus \{j\}}}   
 \binom{d-1}{|u|}^{-1}  
 \left( \var\left\{ \esp\left[g_{u, j}\left(\X_{w_{1}}, \bo{Z}_{\sim w_{1}}\right) | \X_{w_{1}}, \bo{Z}_{\boldsymbol{\varpi}},  Z_{j} \right] \right\}  \right.   \nonumber    \\
& &  \left.  -    
\var\left\{ \esp\left[g_{u, j}\left(\X_{w_{1}}, \bo{Z}_{\sim w_{1}}\right)| \X_{w_{1}}, \bo{Z}_{\boldsymbol{\varpi}} \right] \right\}     \right)   \, ,           \nonumber           
\end{eqnarray}    
because it is shown in \cite{lamboni25math} that         
$$
\esp\left[\M(\bo{\X})  | \bo{\X}_{u \cup j} \right] 
=
\esp\left[g_{u, j}\left(\X_{w_{1}}, \bo{Z}_{\sim w_{1}}\right) | \X_{w_{1}}, \bo{Z}_{\boldsymbol{\varpi}},  Z_{j} \right]  \, ,  
$$    
$$
\esp\left[\M(\bo{\X})  | \bo{\X}_{u} \right]  
=
\esp\left[g_{u, j}\left(\X_{w_{1}}, \bo{Z}_{\sim w_{1}}\right) | \X_{w_{1}}, \bo{Z}_{\boldsymbol{\varpi}} \right] \, .   
$$    

Based on these elements, the relationships between the DSIs and Shapley effects are given below. 
\begin{theorem} \label{theo:ubone}
Assume that $\bo{\X}$ consists of one group of $d$ dependent variables. Then,   
$$
DS_j \leq Sh_j \leq DS_{T_j}, \quad \forall\, j \in \pi_2  \, .          
$$  
\end{theorem}
\begin{preuve}
See  Appendix \ref{app:theo:ubone}.     
\hfill $\square$  
\end{preuve}           
  
It comes out that DSIs from \cite{lamboni26uq} bracket Shapley effects in the case of one group of dependent variables, likewise Sobol' indices bracket Shapley effects for independent variables.

\subsubsection{Case of a tuple of $K$ independent random sets} \label{sec:link2} 
To express the Shapley effect of $\X_{j}$ in the case of $K>1$ groups of independent variables, assume without loss of generality that $j \in \pi_{k^*}$ with $k^* \in \{1, \ldots, K\}$. All of the subsets $u \subseteq  \{1, \ldots, d \} \setminus \{j\}$ can be  written as follows: 
$$
\left\{u\, :  u \subseteq  \{1, \ldots, d \} \setminus \{j\}   \right\} =
\left\{ \{u_1, \ldots, u_K\}\, :  u_k \subseteq  (\pi_{k} \setminus \{j\} ),\forall \; k \in \{1, \ldots, K\}  \right\}  \, ,    
$$      
leading to 
\begin{eqnarray} 
Sh_{j} &=&  \frac{1}{d \Sigma} \sum_{\substack{u_k \subseteq \pi_{k}\\ \forall\,  k \neq k^*}}
 \sum_{\substack{u_{k^*} \subseteq \left(\pi_{k^*} \setminus \{j\}  \right)}} 
\binom{d-1}{|u_1|+ \ldots + |u_K|}^{-1}      \nonumber \\
& & \times \left( \var\left\{ \esp\left[\M\left(\bo{\X}\right) | \bo{\X}_{u_{k^*}\cup \{j\}}, \bo{\X}_{u_{\sim k^*}} \right] \right\} 
- \var\left\{ \esp\left[\M\left(\bo{\X}\right) | \bo{\X}_{u_{k^*}}, \bo{\X}_{u_{\sim k^*}} \right] \right\}  \right) \, , \nonumber           
\end{eqnarray}             
with  $\bo{\X}_{u_{\sim k^*}} := (\bo{\X}_{u_\ell}, \, \forall\, \ell \in \{1, \ldots, K\} \setminus \{k^*\})$. \\   
 
The above expression of $Sh_j$ relies on $\M(\bo{\X})$, and it is worth deriving an equivalent expression using the ER $g(\cdot)$. Given a permutation $\bo{w}_{k}$ of $\pi_k$ such that $u_k=\{w_{1,k}, \ldots, w_{p_{j_k}-1,k}\}$, consider $\boldsymbol{\varpi_k} := (w_{2, k}, \ldots,  w_{p_{j_k}-1,k})$ for any $k \in \{2, \ldots, K\}$, leading to $u_k =\{w_{1, k}, \boldsymbol{\varpi}_k\}$ with $u_k=w_{1, k}$ when $|u_k|=1$.    
When $k=k^*$, consider a permutation $\bo{w}_{k^*}$ of $\pi_{k^*}$ such that $p_{j_{k^*}}$ is the position of $j_{k^*} =j$ in that permutation, that is, $w_{p_{j_{k^*}}, k^*} =j$ and $\bo{w}_{\sim p_{j}} := (w_{p_{j}+1,k^*}, \ldots, w_{d_{k^*},k^*})$.  \\     
  
 For concise notations, denoting $\bo{Z}_{\boldsymbol{\varpi}_k} := (Z_{\ell}, \, \forall\, \ell \in \{\boldsymbol{\varpi}_k\})$;
$\bo{Z}_{\sim \boldsymbol{\varpi}_k} := \left(Z_{\ell}, \, \forall\, \ell \in \{\pi_k \setminus u_k\}\right)$ leads to the partition $\bo{Z}_{\sim w_{1, k}} = \left(\bo{Z}_{\boldsymbol{\varpi}_k}, \bo{Z}_{\sim \boldsymbol{\varpi}_k} \right)$. Define   
$$
\XE_{u_k} := (\X_{w_{1,k}}, \bo{Z}_{\boldsymbol{\varpi}_k}), \; k=2, \ldots, K \, ,    
$$  
$$
\bo{Z}_{\sim \boldsymbol{\varpi}_{k^*}} := \left(\bo{Z}_{\sim \boldsymbol{\varpi}_k}, \forall\, k \in \{2\, , \ldots, K\} \setminus \{k^*\}  \right)  \, .        
$$  
        
Based on such notations, the ER in Equation (\ref{eq:indeg}) can be written as follows:
\begin{equation} \label{eq:erkstar}
g\left(\bo{\X}_{\pi_1},  \bo{\X}_{\Select}, \bo{Z} \right)  
 =   g\left(\bo{\X}_{\pi_1}, \XE_{u_{2}}, \ldots, \XE_{u_{K}}, Z_j, \bo{Z}_{\bo{w}_{\sim p_{j}}}, \, \bo{Z}_{\sim \boldsymbol{\varpi}_{k^*}}  \right)  \, .       
\end{equation}    
   
Using the identities shown in \cite{lamboni25math},  that is,   
$$
\esp\left[\M\left(\bo{\X}\right) | \bo{\X}_{u_{k^*}}, \bo{\X}_{u_{\sim k^*}} \right] =
\esp\left[ g\left(\bo{\X}_{\pi_1},  \bo{\X}_{\Select}, \bo{Z} \right) | \bo{\X}_{u_1}, \XE_{u_2}, \ldots, \XE_{u_K} \right] \, , 
$$   
$$   
\esp\left[\M\left(\bo{\X}\right) | \bo{\X}_{u_{k^*}\cup \{j\}}, \bo{\X}_{u_{\sim k^*}} \right] = \esp\left[ g\left(\bo{\X}_{\pi_1},  \bo{\X}_{\Select}, \bo{Z} \right) | \bo{\X}_{u_1}, \XE_{u_2}, \ldots, \XE_{u_K}, Z_j \right]\, ,  
$$
the Shapley effect of $\X_j$ becomes    
\begin{eqnarray} 
& & Sh_{j} =  \frac{1}{d \Sigma} \sum_{\substack{u_k \subseteq \pi_{k}\\ \forall\,  k \neq k^*}}
 \sum_{\substack{u_{k^*} \subseteq \left(\pi_{k^*} \setminus \{j\}  \right)}} 
\binom{d-1}{|u_1|+ \ldots + |u_K|}^{-1} \times     \nonumber \\
& & \left( \var\left\{\esp\left[ g\left(\bo{\X}_{\pi_1},  \bo{\X}_{\Select}, \bo{Z} \right) | \bo{\X}_{u_1}, \XE_{u_2}, \ldots, \XE_{u_K}, Z_j \right] \right\}    
- \var\left\{\esp\left[ g\left(\bo{\X}_{\pi_1},  \bo{\X}_{\Select}, \bo{Z} \right) | \bo{\X}_{u_1}, \XE_{u_2}, \ldots, \XE_{u_K} \right] \right\}  \right) \, . \nonumber              
\end{eqnarray}  
  
Based on such a new formulation of $Sh_{j}$, the following general inequality holds.  
\begin{theorem}  \label{theo:upBgen}
Assume $\bo{\X}$ consists of $K$ groups of independent random sets. Then,    
\begin{equation} \label{eq:upBgen}
  DS_{j}   \leq Sh_{j}   \leq  DS_{T_j} \, .  
\end{equation}    
\end{theorem}     
\begin{preuve}    
See Appendix \ref{app:theo:upBgen}.  
 \hfill $\square$              
\end{preuve} 
     
Now, one can state that DSIs bracket Shapley effects for any set of non-independent variables. 
It turns out that the total dependent  index of $\X_j$ is an upper-bound of the Shapley effect of $\X_j$. 
    
\subsection{Upper-bounds of dependent sensitivity indices}  \label{sec:up}   
Based on Equation (\ref{eq;todsinew}) and the ER   
$ 
g_{u_k,j_k}\left(\X_{w_{1,k}}, \bo{Z}_{\sim w_{1,k}} , \bo{\X}_{\sim \pi_k} \right)
$, 
which includes only blocks of independent variables, we are able to derive the upper-bounds of DSIs. Such derivative-based upper-bounds require the partial derivatives of $g_{u_k,j_k}(\cdot)$ w.r.t. $Z_{j_k}$ for each set $u_k \subseteq \pi_k \setminus \{j_k\}$. By considering the DM of $\bo{\X}_{\pi_k}$ associated with the permutation $\bo{w}_{k} =\left(u_k,\, w_{p_{j_k}},\,  \bo{w}_{\sim p_{j_k}} \right)$, that is,  
\begin{equation} \label{eq;depm}
\bo{\X}_{\sim w_{1,k}} = r_{u_k, j_k} \left(\XE_{u_k}, Z_{j_k}, \bo{Z}_{\bo{w}_{\sim p_{j_k}}} \right)  \, , 
\end{equation} 
denote with $J^{(u_k,j_k)} \in \R^{d_k}$ the vector containing the partial derivatives of $\bo{\X}_{\pi_{k}}$ w.r.t. $Z_{j_k}$. We can check that the entries of $J^{(u_k,j_k)}$ are  given by   
$$  
J^{(u_k,j_k)}_\ell = \left\{\begin{array}{cl}    
0  &   \mbox{if}  \;  \ell \in u_k\\ 
\frac{\partial \X_{w_{|u_k|+\ell, k}}}{\partial z_{j_k}}  &   \mbox{if}  \;  \ell =1, \ldots, d_k- |u_k|\\
\end{array}                 
\right.           \, .       
$$   
By independence, the derivatives of $\bo{\X}_{\sim \pi_k}$ w.r.t. $Z_{j_k}$ vanish.        

\begin{theorem}   \label{theo:upbdst}
Assume that $\M$ and the dependency functions are (weakly) differentiable functions. If $\left|\frac{\partial \M}{x_{j_k}} \right| \leq M_{1,k}$ for any $j_k \in \pi_k$, then $DS_{T_{j_k}}  \leq  D\!U\!B_{j_k}$ with
$$
D\!U\!B_{j_k} :=  \frac{M_{1,k}^2}{2 d_k \Sigma} \sum_{\substack{u_k \subseteq  \pi_{k} \setminus \{j_k\}}}
 \binom{d_k-1}{|u_k|}^{-1} (d_k- |u_k|) \, 
 \esp\left[ \norme{J^{(u_k,j_k)}}^2 \frac{F_{z_{j_k}}\left(Z_{j_k}\right) \left[1- F_{z_{j_k}}\left(Z_{j_k}\right) \right]}{\left(\rho_{z_{j_k}}\left(Z_{j_k}\right) \right)^2} \right]  \, .  
$$       
\end{theorem}    
\begin{preuve}  
See Appendix \ref{app:theo:upbdst}. 
 \hfill $\square$   
\end{preuve}
 
Using Equation (\ref{eq:upBgen}), it follows that $Sh_j \leq DS_{T_j} \leq  D\!U\!B_{j}$.  The upper-bound derived in Theorem \ref{theo:upbdst} relies only on the maximum, absolute value of the components of the usual gradients of models and expectations involving different DMs. Such quantities can often be obtained without requiring much computational resources, even in higher dimensions. Indeed, dimension-free mean squared errors of estimators of gradients, proposed in \cite{lamboni26axioms}, can contribute to significantly reduce the computational cost in higher dimensions. The following Corollaries deal with particular cases.  
    
\begin{corollary}   \label{coro:upbdst1}
Under the conditions of Theorem \ref{theo:upbdst}, if $\bo{\X}_{\pi_k}$ follows the Gaussian distribution, then 
$$
DS_{T_{j_k}}  \leq \frac{M_{1,k}^2}{2 d_k \Sigma} \sum_{\substack{u_k \subseteq  \pi_{k} \setminus \{j_k\}}}
 \binom{d_k-1}{|u_k|}^{-1} (d_k- |u_k|) \norme{J^{(u_k,j_k)}}^2 \, 
 \esp\left[ \frac{F_{z_{j_k}}\left(Z_{j_k}\right) \left[1- F_{z_{j_k}}\left(Z_{j_k}\right) \right]}{\left(\rho_{z_{j_k}}\left(Z_{j_k}\right) \right)^2} \right] \, .   
$$
\end{corollary}   
 \begin{preuve}
It is straightforward, as $J^{(u_k,j_k)}$ is a vector of constants (see \cite{lamboni23axioms}).
\hfill $\square$
\end{preuve}
 
For models of the form $\M_l(\bo{\X}) := \boldsymbol{\beta}_{\pi_k}^\T \bo{\X}_{\pi_k} + \M_{\sim k}(\bo{\X}_{\sim \pi_k})$ with $\M_{\sim k}$ a given function and  $\boldsymbol{\beta}_{\pi_k} := (\beta_{\ell}, \forall \, \ell \in \pi_k) \in \R^{d_k}$, we have $M_{1,k} = \max_{\ell \in \pi_k} \left \{ \left| \beta_{\ell} \right|\right \}$.    
         
\begin{corollary}   \label{coro:upbdst2}
Let $M_{1,\sim u_k} := \sum_{\ell \in (\pi_k\setminus u_k)} \beta_{\ell}^2$. Under the conditions of Theorem \ref{theo:upbdst}, if $\M(\bo{\X})=\M_l(\bo{\X})$ then        
$$ 
DS_{T_{j_k}}  \leq \frac{1}{2 d_k \Sigma} \sum_{\substack{u_k \subseteq  \pi_{k} \setminus \{j_k\}}}
 \binom{d_k-1}{|u_k|}^{-1}  M_{1,\sim u_k} \, 
 \esp\left[ \norme{J^{(u_k,j_k)}}^2  \frac{F_{z_{j_k}}\left(Z_{j_k}\right) \left[1- F_{z_{j_k}}\left(Z_{j_k}\right) \right]}{\left(\rho_{z_{j_k}}\left(Z_{j_k}\right) \right)^2} \right] \, .   
$$               
\end{corollary}     
 \begin{preuve}
See the proof of Theorem \ref{theo:upbdst} in Appendix \ref{app:theo:upbdst}.  
\hfill $\square$ 
\end{preuve}  

\begin{corollary}   \label{coro:upbdst3}
Under the conditions of Corollary \ref{coro:upbdst2}, if $\bo{\X}_{\pi_k}$ follows the Gaussian distribution then
$$ 
DS_{T_{j_k}}  \leq \frac{1}{2 d_k \Sigma} \sum_{\substack{u_k \subseteq  \pi_{k} \setminus \{j_k\}}}
 \binom{d_k-1}{|u_k|}^{-1}  M_{1,\sim u_k} \norme{J^{(u_k,j_k)}}^2 \, 
 \esp\left[ \frac{F_{z_{j_k}}\left(Z_{j_k}\right) \left[1- F_{z_{j_k}}\left(Z_{j_k}\right) \right]}{\left(\rho_{z_{j_k}}\left(Z_{j_k}\right) \right)^2} \right] \, .    
$$                  
\end{corollary}   
   
Such a result is straightforward. Since we are able to choose the distribution of $\Z_{j_k}$, taking $Z_{j_k} \stackrel{d}{=} \X_{j_k}$ yields 
$$ 
DS_{T_{j_k}}  \leq \frac{E_{j_k}}{2 d_k \Sigma} \sum_{\substack{u_k \subseteq  \pi_{k} \setminus \{j_k\}}}
 \binom{d_k-1}{|u_k|}^{-1}  M_{1,\sim u_k} \norme{J^{(u_k,j_k)}}^2 \, , 
$$ 
where $E_{j_k} := \esp\left[ \frac{F_{z_{j_k}}\left(Z_{j_k}\right) \left[1- F_{z_{j_k}}\left(Z_{j_k}\right) \right]}{\left(\rho_{z_{j_k}}\left(Z_{j_k}\right) \right)^2} \right]$. \\ 

Moreover, it is also interesting to have the upper-bouds of DSIs based on (dependent) gradients of functions with non-independent variables (\cite{lamboni23axioms,lamboni26axioms}).  To that end, consider the DM given by Equation (\ref{eq;depm}), and recall the definition of the first-type, total generalized sensitivity indices of $\Z_{j_k}$ for $r_{u_k, j_k}$ as follows (\cite{lamboni11,gamboa14}): 
\begin{eqnarray} 
GSI_{T_{u_k,j_k}} &:=& \frac{\trace\left(\var\left[ r_{u_k, j_k} \left(\XE_{u_k}, Z_{j_k}, \bo{Z}_{\bo{w}_{\sim p_{j_k}}} \right) -\esp_{\Z_{j_k}}\left[r_{u_k, j_k} \left(\XE_{u_k}, Z_{j_k}, \bo{Z}_{\bo{w}_{\sim p_{j_k}}} \right) \right]\right]  \right)}{\trace\left(\var\left[ r_{u_k, j_k} \left(\XE_{u_k}, Z_{j_k}, \bo{Z}_{\bo{w}_{\sim p_{j_k}}} \right)\right]  \right)} \nonumber\\
&=&  \frac{\esp\left[\norme{ r_{u_k, j_k} \left(\XE_{u_k}, Z_{j_k}, \bo{Z}_{\bo{w}_{\sim p_{j_k}}} \right) -r_{u_k, j_k} \left(\XE_{u_k}, Z_{j_k}', \bo{Z}_{\bo{w}_{\sim p_{j_k}}} \right) }^2 \right]}{2 \trace\left(\var\left[ r_{u_k, j_k} \left(\XE_{u_k}, Z_{j_k}, \bo{Z}_{\bo{w}_{\sim p_{j_k}}} \right)\right]  \right)} \, , \nonumber 
\end{eqnarray}
with $Z_{j_k}'$ being an i.i.d.copy of $Z_{j_k}$; $\trace(\cdot)$ being the trace function 
 Also, we use $grad (\M)$ for the dependent gradient of $\M$ (see \cite{lamboni23axioms,lamboni24axioms,lamboni26axioms} for more details).  
        
\begin{theorem}   \label{theo:upbdst2}
Let  $M_{1,k}^d$ be  the maximum, absolute value of the components of $grad (\M)$. If $\M$ and the dependency functions are (weakly) differentiable functions, then 
$$
DS_{T_{j_k}}  \leq \min \left\{ D\!U\!B_{j_k},\,  D\!U\!B_{j_k}' \right\} \, , 
$$
 with     
$$
D\!U\!B_{j_k}' :=  \frac{\left(M_{1,k}^d \right)^2}{ d_k \Sigma} \sum_{\substack{u_k \subseteq  \pi_{k} \setminus \{j_k\}}}
 \binom{d_k-1}{|u_k|}^{-1} \trace\left(\var\left[ r_{u_k, j_k} \left(\XE_{u_k}, Z_{j_k}, \bo{Z}_{\bo{w}_{\sim p_{j_k}}} \right)\right]  \right)  GSI_{T_{u_k,j_k}}   \, .   
$$         
\end{theorem}      
\begin{preuve}
See Appendix \ref{app:theo:upbdst2}.   
\hfill $\square$     
\end{preuve} 
 
Note that only $M_{1,k}^d$ requires model evaluations in the expression of $D\!U\!B_{j_k}'$.
             
\section{An analytical test case} \label{sec:test} 
Consider the linear function with three inputs given by      
\begin{equation} \label{eq:line}         
    \M(\bo{\X}) = \X_1 + \X_2 +  \X_3  \, ,     \nonumber    
	\quad
	\bo{\X} \sim \mathcal{N}\left(\bo{0},\, \left[\begin{array}{ccc} 
\sigma_1^2 &\rho_{12} \sigma_1\sigma_2 & \rho_{13}\sigma_1\sigma_3 \\ 
\rho_{12}\sigma_1\sigma_2 & \sigma_2^2 & \rho_{23}\sigma_2\sigma_3  \\
\rho_{13}\sigma_1\sigma_3 & \rho_{23}\sigma_2\sigma_3  & \sigma_{3}^2    
\end{array}\right]\right) \, . 
\end{equation}   
 One of the dependency models is given by (\cite{lamboni21,lamboni25math})
\begin{eqnarray}
 \left\{\begin{array}{ccl}  
    X_2 &=& \frac{\rho_{12}\sigma_2}{\sigma_1} X_1 + \sqrt{1-\rho_{12}^2}  Z_2   \\
		X_3 &=& \frac{\rho_{13}\sigma_3}{\sigma_1} X_1 + \frac{\sigma_3(\rho_{23}-\rho_{12}\rho_{13})}{\sigma_2 \sqrt{1-\rho_{12}^2}} Z_2 + \sqrt{\frac{1-\rho_{12}^2-\rho_{13}^2-\rho_{23}^2+2\rho_{12}\rho_{13}\rho_{23}}{1-\rho_{12}^2}} Z_3
      \end{array} \right.   \, ,   \nonumber                         
\end{eqnarray}     
where $Z_j \sim \mathcal{N}\left(0, \sigma_j^2 \right)$, $j=2,\,3$ and $Z_2, \, Z_3,\, X_1$ are independent. The DSIs are derived in \cite{lamboni26uq}, and such indices are equal to the Shapley effects.\\    

The above DM is associated with $\bo{w} =(1, 2, 3)$, leading to $p_1=1$ and $u_1 =\emptyset$ for the input $\X_1$. Concerning $\X_2$, one can see that $p_2=2$ and $u_2 =\{1\}$, while $p_3=3$ and $u_3 =\{1, 2\}$ for the input $\X_3$. Thus,  
$ 
J^{(\emptyset, 1)} = \left[1, \, \frac{\rho_{12}\sigma_2}{\sigma_1}, \,  \frac{\rho_{13}\sigma_3}{\sigma_1} \right]^\T       
$ for $\X_1$,   
$
J^{(\{1\}, 2)} = \left[0, \, \sqrt{1-\rho_{12}^2}, \,   \frac{\sigma_3(\rho_{23}-\rho_{12}\rho_{13})}{\sigma_2 \sqrt{1-\rho_{12}^2}}  \right]^\T$ for $\X_2$ and 
$
J^{(\{1,2\}, 3)} = \left[0, \, 0, \, \sqrt{\frac{1-\rho_{12}^2-\rho_{13}^2-\rho_{23}^2+2\rho_{12}\rho_{13}\rho_{23}}{1-\rho_{12}^2}}  \right]^\T     
$ for $\X_3$.   All of the values of $J^{(u_j, j)}$ are provided in Appendix \ref{app:Jval}. 
As $d=3$ and $M_{1} =1$,  the upper-bound of $DST_{j}$ is     
$$  
D\!U\!B_{j} :=  \frac{E_{j}}{6 \Sigma} \sum_{\substack{u_j \subseteq  \{1, 2, 3\} \setminus \{j\}}}
 \binom{d-1}{|u_j|}^{-1} (d- |u_j|)  \norme{J^{(u_j,j)}}^2   \, ;     
$$   
leading to the following upper-bounds of $\X_1$, $\X_2$ and $\X_3$, respectively,
$$
D\!U\!B_{1} :=  \frac{E_1}{6 \Sigma} \left( 3 \norme{J^{(\emptyset,1)}}^2 
+  \norme{J^{(\{2\},1)}}^2 +  \norme{J^{(\{3\},1)}}^2 + \norme{J^{(\{23\},1)}}^2 \right) \, ,   
$$   
$$  
D\!U\!B_{2} :=  \frac{E_2}{6 \Sigma} \left( 3 \norme{J^{(\emptyset,2)}}^2 
+  \norme{J^{(\{1\},2)}}^2 +  \norme{J^{(\{3\},2)}}^2 + \norme{J^{(\{13\},2)}}^2 \right) \, ,
$$  
$$
D\!U\!B_{3} :=  \frac{E_3}{6 \Sigma} \left( 3 \norme{J^{(\emptyset,3)}}^2 
+  \norme{J^{(\{1\},3)}}^2 +  \norme{J^{(\{2\},3)}}^2 + \norme{J^{(\{12\},3)}}^2 \right) \, .  
$$    
       
 Figure \ref{fig:gsi1} depicts the DSIs and their upper-bounds using $\sigma_1^2 =2, \, \sigma_2^2=\sigma_3^2=8$ and the sets of correlations listed in Table \ref{tab:text1}.         
\begin{table}[ht]        
\centering
\begin{tabular}{lccc c lccc}
  \hline  
  \hline
  & $\rho_{12}$ & $\rho_{13}$ & $\rho_{23}$ & &  & $\rho_{12}$ & $\rho_{13}$ & $\rho_{23}$ \\ 
  \hline
 $C_1$ & 1 & 1 & 1     & &  $C_6$ & 0.000 & 0.600 & 0.000  \\ 
 $C_2$ & 0.250 & 0.500 & 0.750       & &  $C_7$ & 0.000 & 0.000 & 0.000\\  
 $C_3$ &  0.010 & 0.000 & 0.750      & &  $C_8$ & 0.250 & 0.800 & 0.500 \\ 
 $C_4$ & 0.500 & 0.500 & 0.500       & &  $C_9$ & 0.000 & 0.750 & 0.450 \\ 
 $C_5$ & -0.500 & 0.500 & -0.500     & &  $C_{10}$ & -0.250 & 0.250 & 0.250 \\ 
   \hline
	 \hline 
\end{tabular}
\caption{Ten sets of the correlation values.}   
\label{tab:text1}    
\end{table}
 \begin{figure}[!hbp]   
\begin{center}  
\includegraphics[height=8cm,width=16cm,angle=0]{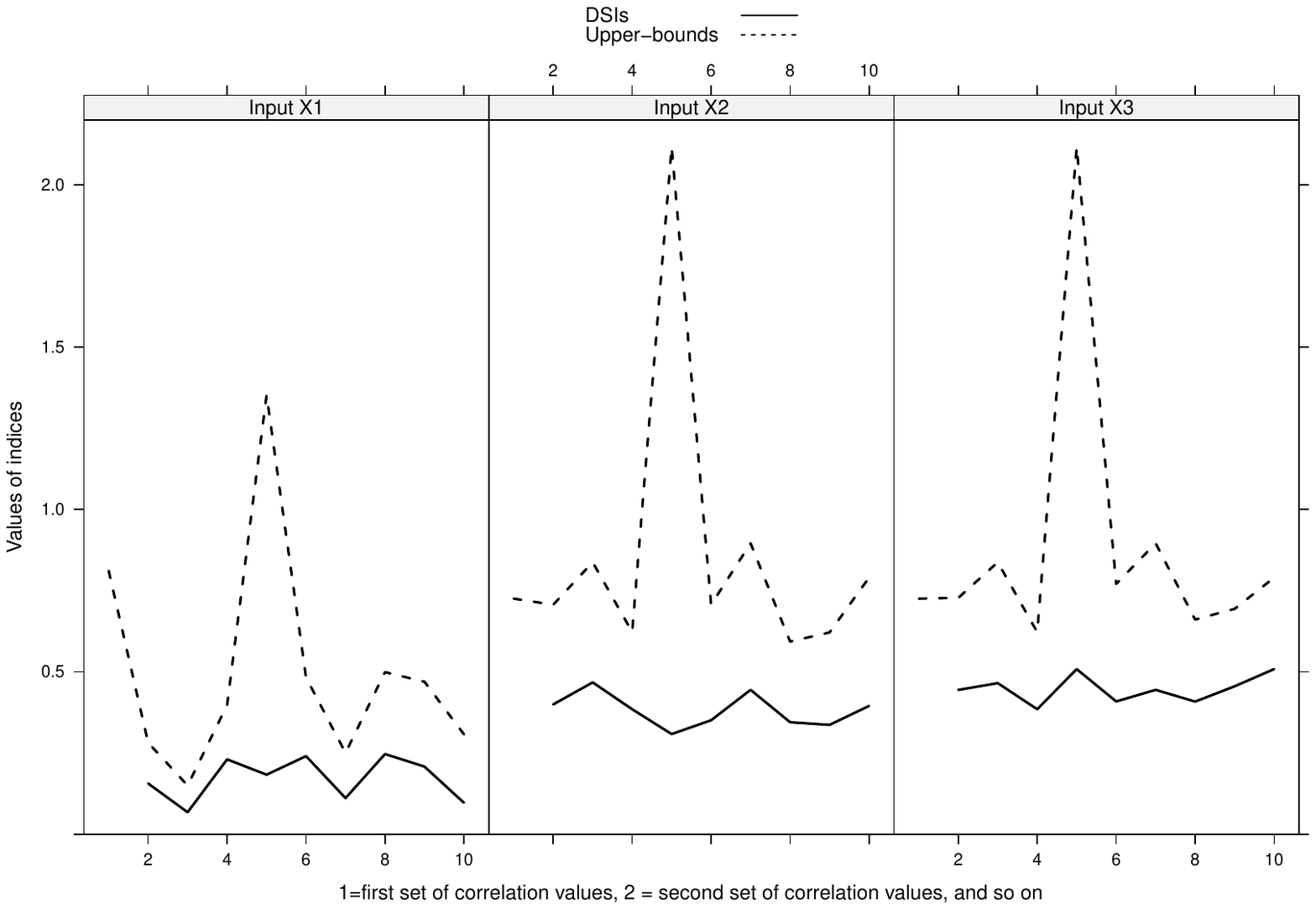}
\end{center} 
\caption{DSIs and their upper-bounds.}       
 \label{fig:gsi1}             
\end{figure}     
   
It comes out from Figure \ref{fig:gsi1} that $D\!U\!B_{j}$s are upper-bounds of DSIs (as expected). Moreover, for some sets of correlations, the values of upper-bounds are close to those of DSIs. Consequently, such upper-bounds can serve as screening measures for models with non-independent variables.\\

It is worth noting that in the presence of perfect correlations ($\rho_{12}=1$),  one obtains a degenerate case, and the following DM must be considered to avoid infinity values: 
$$
\X_{2}  = \frac{\rho_{12}\sigma_2}{\sigma_1} X_1,      
\qquad    \quad     
\X_{3}  = \frac{\rho_{13}\sigma_3}{\sigma_1} X_1 + \sqrt{1-\rho_{13}^2} Z_3 \, , 
$$      
meaning that the coefficient of $Z_2$ vanishes. Additionally, if $\rho_{13}=1$, then one has 
$$
\X_{2}  = \frac{\rho_{12}\sigma_2}{\sigma_1} X_1, 
\qquad    \quad         
\X_{3}  = \frac{\rho_{13}\sigma_3}{\sigma_1} X_1\, .  
$$    
                                                   
\section{Conclusion} \label{sec:con}   
Firstly, we have shown that Shapley effects of non-independent variables range between DSIs, that is, $DS_j \leq Sh_j \leq DS_{T_j}$ for every model evaluated at non-independent variables.  The equality holds for linear models evaluated at a Gaussian random vector, yielding sharp inequalities between Shapley effects and DSIs of non-independent variables.  For independent variables, DSIs boil down to Sobol' indices, and sharp inequalities hold between Shapley effects and Sobol' indices as well (\cite{owen17}). 
 Regarding the computations of both types of indices, it appears that DSIs require a number of model runs that is less than that of Shapley effects. Thus, the total DSIs can also be used for screening  Shapley effects.\\ 
    
Secondly, we have provided the upper-bouds of dependent sensitivity indices that require fewer model evaluations compared to the computations of DSIs. Such upper-bounds serve as new screening measures for models with non-independent variables, and they can help sometimes for obtaining practical estimates of the total DSIs.  While an analytical test case is used to show the effectiveness of the proposed upper-bounds, the corresponding screening measures are going to be more efficient for linear models and for models with small values of the traditional gradients or dependent gradients. Despite such upper-bounds are valid for every smooth model, more investigations are needed so as to come out with a sharp inequality between the total DSI and its upper-bound. 
                             
 \section*{Acknowledgements}   
 We would like to thank the two referees for their comments and suggestions that have helped improving this paper.   
      
\begin{appendices}            
\section{Proof of Theorem \ref{theo:ubone}} \label{app:theo:ubone}
As the ER $g\left(\X_{w_{1}}, \bo{Z}_{\sim w_{1}}\right)$ associated with $\bo{w}$ includes only independent variables,  the Hoeffding decomposition is given as follows:      
$$
\Sigma = \sum_{\substack{r \subseteq  \pi_{2} =\{w_1, \boldsymbol{\varpi}, j , \bo{w}_{\sim p_j} \}\\ |r| > 0}} \Sigma_{r} \, ,   
$$
and we can check that  
$$
A_1 := \var\left\{ \esp\left[ g\left(\X_{w_{1}}, \bo{Z}_{\sim w_{1}}\right) | \X_{w_{1}}, \bo{Z}_{\boldsymbol{\varpi}} \right] \right\} 
= 
\sum_{\substack{ r \subseteq \{w_1, \boldsymbol{\varpi} \} \\ |r| > 0 }} \Sigma_{r}  \, ,    
$$  
$$ 
A_2 :=   \var\left\{ \esp\left[ g\left(\X_{w_{1}}, \bo{Z}_{\sim w_{1}}\right) | \X_{w_{1}}, \bo{Z}_{\boldsymbol{\varpi}},  Z_{j}  \right] \right\}  
= \sum_{\substack{ r \subseteq \{w_1, \boldsymbol{\varpi}, j \} \\ |r| > 0 }} \Sigma_{r}
  \, .     
$$        
leading to     
$$
\Sigma_j \leq  
A_2 - A_1 = \sum_{\substack{r \subseteq \{w_1, \boldsymbol{\varpi} \} }} \Sigma_{r \cup \{j\}} = \sum_{\substack{r \subseteq \{w_1, \boldsymbol{\varpi}, j \} \\ j \in r }} \Sigma_{r} \leq  \sum_{\substack{r \subseteq \pi_2 \\ j \in r }} \Sigma_{r}
   \, .        
$$    
The result holds because $\Sigma_j =\sigma_{u, j}^{fo}$ and $\sigma_{u, j}^{tot} =\sum_{\substack{r \subseteq \pi_2 \\ j \in r }} \Sigma_{r}$.

\section{Proof of Theorem \ref{theo:upBgen}} \label{app:theo:upBgen}
Firstly, consider the ER given by Equation (\ref{eq:erkstar}), that is, 
$$
 g\left(\bo{\X}_{\pi_1}, \XE_{u_{2}}, \ldots, \XE_{u_{K}}, Z_j, \bo{Z}_{\bo{w}_{\sim p_{j}}}, \, \bo{Z}_{\sim \boldsymbol{\varpi}_{k^*}}  \right) \stackrel{d}{=}  \M(\bo{\X}) \, .  
$$         
For the sake of simplicity, let $\XE_{u_1} :=\bo{\X}_{u_1}$, $\XE_{u_{K+1}} := Z_j$ and $\XE_{u_{K+2}} :=\left\{\bo{Z}_{\bo{w}_{\sim p_{j}}}, \, \bo{Z}_{\sim \boldsymbol{\varpi}_{k^*}} \right\}$ be a random set containing the other inputs of $g(\cdot)$. By using $\bo{\XE} :=(\XE_{u_1}, \ldots, \XE_{u_{K+2}})$, the Hoeffding decomposition of $g(\bo{\XE})$ is given by    
$$ 
g(\bo{\XE}) = \esp\left[g(\bo{\XE})\right] +  \sum_{\substack{v \subseteq \left\{\{u_1\}, \ldots, \{u_{K+2}\} \right\}\\ |v|>0}} g_{v}\left( \bo{\XE}_v \right);     
\qquad   
\Sigma =  \sum_{\substack{v \subseteq \left\{\{u_1\}, \ldots, \{u_{K+2}\} \right\}\\ |v|>0}} \Sigma_{v}  \, .      
$$  
By analogy to $v \subseteq \{1, \ldots, d\}$, it is to be noted that  
$$\
v \subseteq \{ \{u_{1}\}, \ldots, \{u_{K}\} \} = v \in \left\{ \{u_{1}\}, \ldots, \{u_{K}\}, \{u_{1}, u_{2} \}, \ldots, \{u_{K-1}, u_{K} \}, \ldots, \{u_{1}, \ldots, u_{K} \}\right\} \, .   
$$             
We can then check that      
$$   
B_1 := \var\left\{\esp\left[ g\left(\bo{\X}_{\pi_1},  \bo{\X}_{\Select}, \bo{Z} \right) | \bo{\X}_{u_1}, \XE_{u_2}, \ldots, \XE_{u_K}, Z_j \right] \right\}  =  \sum_{\substack{v \subseteq \left\{\{u_1\}, \ldots, \{u_{K+1}\} \right\}\\ |v| >0}} \Sigma_{v}  \, ,   
$$ 
$$
B_2 := \var\left\{\esp\left[ g\left(\bo{\X}_{\pi_1},  \bo{\X}_{\Select}, \bo{Z} \right) | \bo{\X}_{u_1}, \XE_{u_2}, \ldots, \XE_{u_K} \right] \right\}  =  \sum_{\substack{v \subseteq \left\{\{u_1\}, \ldots, \{u_{K}\} \right\}\\ |v|>0}} \Sigma_{v} \, ,
$$ 
leading to 
$$
B_1-B_2 = \sum_{\substack{v \subseteq \left\{\{u_1\}, \ldots, \{u_{K}\} \right\}}} \Sigma_{v \cup \{u_{K+1}\}} \, .  
$$         
Secondly, we have 
$$
\sigma_{u_{k^*},j}^{fo} 
  \leq 
\sum_{\substack{v \subseteq \left\{\{u_1\}, \ldots, \{u_{K}\} \right\}}} \Sigma_{v \cup \{u_{K+1}\}}   \leq \sigma_{u_{k^*},j}^{tot}   \, ,       
$$
because    
$$
\sigma_{u_{k^*},j}^{fo} =  \var\left\{\esp\left[g\left(\bo{\XE} \right) \,| Z_j \right] \right\}  = \var\left\{\esp\left[g\left(\bo{\XE} \right) \,| \XE_{u_{K+1}} \right] \right\} = \Sigma_{u_{K+1}} \, ,   
$$
and     
$$
\sigma_{u_{k^*},j}^{tot} =  \sum_{\substack{v \subseteq \left\{\{u_1\}, \ldots, \{u_{K}\}, \{u_{K+2}\} \right\}}} \Sigma_{v \cup \{u_{K+1}\}} \, .      
$$ 
Therefore, the following inequalities hold:   
\begin{equation}  
Sh_{j} \geq  \frac{1}{d \Sigma} \sum_{\substack{u_k \subseteq \pi_{k}\\ \forall\,  k \neq k^*}}
 \sum_{\substack{u_{k^*} \subseteq \left(\pi_{k^*} \setminus \{j\}  \right)}}   
\binom{d-1}{|u_1|+ \ldots + |u_K|}^{-1} \sigma_{u_{k^*},j}^{fo}  \, ,  \nonumber 
\end{equation}
\begin{equation}   
Sh_{j} \leq  \frac{1}{d \Sigma} \sum_{\substack{u_k \subseteq \pi_{k}\\ \forall\,  k \neq k^*}}
 \sum_{\substack{u_{k^*} \subseteq \left(\pi_{k^*} \setminus \{j\}  \right)}}   
\binom{d-1}{|u_1|+ \ldots + |u_K|}^{-1} \sigma_{u_{k^*},j}^{tot}  \, .    \nonumber    
\end{equation}    
    
Using $A_{k^*} := \sum_{\substack{u_k \subseteq \pi_{k}\\ \forall\,  k \neq k^*}} \binom{d-1}{|u_1|+ \ldots + |u_K|}^{-1}$ leads to
\begin{equation}  \label{eq:ubshj}  
\frac{1}{d \Sigma} \sum_{\substack{u_{k^*} \subseteq \left(\pi_{k^*} \setminus \{j\}  \right)}} A_{k^*} \, \sigma_{u_{k^*},j}^{fo} 
\leq    
Sh_{j} 
\leq 
\frac{1}{d \Sigma} \sum_{\substack{u_{k^*} \subseteq \left(\pi_{k^*} \setminus \{j\}  \right)}} A_{k^*} \, \sigma_{u_{k^*},j}^{tot}  \, . 
\end{equation}  
Thirdly, consider $\Gamma_{k^*} := \{0, \ldots, d_{k^*}-1 \}$;  $\Gamma_k := \{0, \ldots, d_k \}$ for any $k  \in \{1, \ldots, K\} \setminus\{k^*\}$ ; $\Gamma := \bigtimes_{k=1}^K \Gamma_k$,  $\bo{q} :=(q_1, \ldots, q_K) \in \Gamma$, $\bo{q}_{\sim k^*} :=(q_1, \ldots,  q_{k^* -1}, q_{k^* +1} \ldots, q_K)$ and $||\bo{q}||_1 := q_1+ \ldots+q_K$. 
Keeping in mind the Vandermonte convolution of binomial coefficients, one can write  
\begin{eqnarray}         
A_{k^*}  &=& \sum_{u_1 \subseteq \pi_{1}} \ldots \sum_{u_{k^*-1} \subseteq \pi_{{k^*-1}}}
\sum_{u_{k^*+1} \subseteq \pi_{{k^*+1}}} \ldots
\sum_{u_{K} \subseteq \pi_{{K}}} \binom{d-1}{|u_1|+ \ldots + |u_K|}^{-1}   \nonumber \\
&=&  \sum_{q_1 \in \Gamma_1} \binom{d_1}{q_1} \ldots \sum_{q_{k^*-1} \in \Gamma_{k^*-1}}
\binom{d_{k^*-1}}{q_{k^*-1}}
\sum_{q_{k^*+1} \in \Gamma_{k^*+1}} \binom{d_{k^*+1}}{q_{k^*+1}}  \ldots \sum_{q_{K} \in \Gamma_{K}}  
\binom{d_{K}}{q_{K}}  \nonumber \\    
& & \times  \binom{d-1}{|u_{k^*}| + ||\bo{q}_{\sim k^*}||_1 }^{-1}    \nonumber \\     
&=& \sum_{\substack{q_k \in \Gamma_k \\ \forall\,  k \neq k^*}} \prod_{\substack{k=1 \\ k\neq k^*}}^K  \binom{d_k}{q_k}  \binom{d-1}{|u_{k^*}| +  ||\bo{q}_{\sim k^*}||_1}^{-1}  
= \sum_{m=0}^{d-d_{k*}}   \sum_{\substack{q_k \in \Gamma_k \\ \forall\,  k \neq k^*\\ \substack{||\bo{q}_{\sim k^*}||_1 =m}}} \prod_{\substack{k=1 \\  k\neq k^*}}^K  \binom{d_k}{q_k} \binom{d-1}{|u_{k^*}| + m}^{-1}       
\nonumber \\  
& =&  \sum_{m=0}^{d-d_{k*}} \binom{d-d_{k^*}}{m}  \binom{d-1}{|u_{k^*}| + m}^{-1}  \nonumber \\          
& = &  \binom{d_{k^*}-1}{|u_{k^*}|}^{-1}     \sum_{m=0}^{d-d_{k*}} \binom{d-d_{k^*}}{m}  \binom{d_{k^*}-1}{|u_{k^*}|}  \binom{d-1}{|u_{k^*}| + m}^{-1} =:  \binom{d_{k^*}-1}{|u_{k^*}|}^{-1} A_{4}   \nonumber   \, ,                     
\end{eqnarray}    	 	   
where  
$
A_{4} :=    \sum_{m=0}^{d-d_{k*}} \binom{d-d_{k^*}}{m}  \binom{d_{k^*}-1}{|u_{k^*}|}  \binom{d-1}{|u_{k^*}| + m}^{-1} \, .  
$.
By expressing the binomial coefficients involved in $A_{4}$, and by re-organizing the factorial terms, one can write
\begin{eqnarray} 
A_{4} &=&  \sum_{m=0}^{d-d_{k*}} \binom{d-1 - |u_{k^*}|-m}{d_{k^*}-1 -|u_{k^*}|}  \binom{|u_{k^*}| +m}{|u_{k^*}|}  \binom{d-1}{d_{k^*}-1}^{-1} \nonumber \\   
&=& \binom{d-1}{d_{k^*}-1}^{-1} \; \sum_{k=|u_{k^*}|}^{d-d_{k* + |u_{k^*}|}} \binom{d-1 -k}{d_{k^*}-1 -|u_{k^*}|}  \binom{k}{|u_{k^*}|}   \nonumber \\  
&=&   \binom{d-1}{d_{k^*}-1}^{-1}  \binom{d}{d_{k^*}}  =  \frac{d}{d_{k^*}} \, , \nonumber 
\end{eqnarray} 
using a kind of the generalized Hockey-Stick identity.\\      
Finally, combining all these elements yields  
$
A_{k^*}  = \binom{d_{k^*}-1}{|u_{k^*}|}^{-1} \frac{d}{d_{k^*}}
$,   
and  Equation (\ref{eq:ubshj}) becomes 
$$
  \frac{1}{d_{k^*} \Sigma} 
 \sum_{\substack{u_{k^*} \subseteq \left(\pi_{k^*} \setminus \{j\}  \right)}}  
\binom{d_{k^*}-1}{|u_{k^*}|}^{-1}  
 \sigma_{u_{k^*},j}^{fo}  = DS_j   
\leq
 Sh_{j} 
 \leq  DS_{T_j} \, . 
$$

\section{Proof of Theorem \ref{theo:upbdst}} \label{app:theo:upbdst}
Using the ER of $\M(\bo{\X})$, that is,     
$$  
g_{u_k,j_k}\left(\X_{w_{1,k}}, \bo{Z}_{\sim w_{1,k}} , \bo{\X}_{\sim \pi_k} \right) := 
\M\left(\X_{w_{1,k}}, r_{u_k, j_k} \left(\XE_{u_k}, Z_{j_k}, \bo{Z}_{\bo{w}_{\sim p_{j_k}}} \right), \, \bo{\X}_{\sim \pi_k}  \right) \, ,      
$$   
we can check that  (\cite{lamboni21,lamboni23axioms})       
$$ 
\frac{\partial g_{u_k,j_k}}{\partial z_{j_k}}\left(\X_{w_{1,k}}, \bo{Z}_{\sim w_{1,k}} , \bo{\X}_{\sim \pi_k} \right) = \left( J^{(u_k,j_k)} \right)^\T \nabla \M\left(\X_{w_{1,k}}, r_{u_k, j_k} \left(\XE_{u_k}, Z_{j_k}, \bo{Z}_{\bo{w}_{\sim p_{j_k}}} \right), \, \bo{\X}_{\sim \pi_k}  \right)  \,   .   
$$     
     
Using the inequality $\sigma_{u_k, j_k}^{tot}  \leq   \frac{1}{2} \esp\left[ \left(\frac{\partial g}{\partial z_{j_k}} \right)^2 \frac{F_{z_{j_k}}\left(Z_{j_k}\right) \left[1- F_{z_{j_k}}\left(Z_{j_k}\right) \right]}{\left(\rho_{z_{j_k}}\left(Z_{j_k}\right) \right)^2} \right]$ shown in \cite{lamboni22b}, one can write  
\begin{eqnarray} 
\sigma_{u_k, j_k}^{tot} & \leq & \frac{1}{2}   \esp\left[ \left(\norme{J^{(u_k,j_k)}} \norme{\nabla_{\sim u_k} \M} \right)^2 \frac{F_{z_{j_k}}\left(Z_{j_k}\right) \left[1- F_{z_{j_k}}\left(Z_{j_k}\right) \right]}{\left(\rho_{z_{j_k}}\left(Z_{j_k}\right) \right)^2} \right] \nonumber \\ 
 & \leq &  \frac{(d_k- |u_k|) M_{1,k}^2 }{2} \,   
 \esp\left[ \norme{J^{(u_k,j_k)}}^2 \frac{F_{z_{j_k}}\left(Z_{j_k}\right) \left[1- F_{z_{j_k}}\left(Z_{j_k}\right) \right]}{\left(\rho_{z_{j_k}}\left(Z_{j_k}\right) \right)^2} \right]  \, ,   \nonumber  
\end{eqnarray}  
and the result holds.

\section{Proof of Theorem \ref{theo:upbdst2}} \label{app:theo:upbdst2}
Firstly, using $Z_{j_k}'$ for an i.i.d. copy of $Z_{j_k}$, and keeping in mind that 
$$  
g_{u_k,j_k}\left(\X_{w_{1,k}}, \bo{Z}_{\sim w_{1,k}} , \bo{\X}_{\sim \pi_k} \right) := 
\M\left(\X_{w_{1,k}}, r_{u_k, j_k} \left(\XE_{u_k}, Z_{j_k}, \bo{Z}_{\bo{w}_{\sim p_{j_k}}} \right), \, \bo{\X}_{\sim \pi_k}  \right) \, ,     
$$   
the total-effect variance of $\Z_{j_k}$ is given by      
\begin{eqnarray}
 \sigma_{u_k, j_k}^{tot}  
	&=& \frac{1}{2}   \esp\left\{ \left[\M\left(\X_{w_{1,k}}, r_{u_k, j_k} \left(\XE_{u_k}, Z_{j_k}, \bo{Z}_{\bo{w}_{\sim p_{j_k}}} \right), \, \bo{\X}_{\sim \pi_k}  \right)  \right. \right.  \nonumber \\
& & \left. \left. \hfill 
-   
\M\left(\X_{w_{1,k}}, r_{u_k, j_k} \left(\XE_{u_k}, Z_{j_k}', \bo{Z}_{\bo{w}_{\sim p_{j_k}}} \right), \, \bo{\X}_{\sim \pi_k}  \right)
\right]^2 \right\} \nonumber \\  
 & \leq & \left(M_{1,k}^d\right)^2 \; \frac{1}{2} 
 \esp\left\{\left[ \norme{ r_{u_k, j_k} \left(\XE_{u_k}, Z_{j_k}, \bo{Z}_{\bo{w}_{\sim p_{j_k}}} \right)
-   
 r_{u_k, j_k} \left(\XE_{u_k}, Z_{j_k}', \bo{Z}_{\bo{w}_{\sim p_{j_k}}} \right)}^2 \right] \right\}  \nonumber \\
&=&  \left(M_{1,k}^d\right)^2 \trace\left(\var\left[ r_{u_k, j_k} \left(\XE_{u_k}, Z_{j_k}, \bo{Z}_{\bo{w}_{\sim p_{j_k}}} \right)\right]  \right)  GSI_{T_{u_k,j_k}} \, , \nonumber 
\end{eqnarray} 	         
keeping in mind the Taylor expansion of functions with non-independent variables (see \cite{sommer20,lamboni23axioms}).

\section{Derivatives of dependency models used in Section \ref{sec:test}} \label{app:Jval}
  
Likewise, we can check that 
$
J^{(\{1\}, 3)} = \left[0, \,   \frac{\sigma_2(\rho_{23}-\rho_{12}\rho_{13})}{\sigma_3 \sqrt{1-\rho_{13}^2}}, \, \sqrt{1-\rho_{13}^2} \right]^\T$;    
 $ 
J^{(\emptyset, 2)} = \left[ \frac{\rho_{12}\sigma_1}{\sigma_2}, \, 1, \,  \frac{\rho_{23}\sigma_3}{\sigma_2} \right]^\T    
$;      
$
J^{(\{2\}, 1)} = \left[\sqrt{1-\rho_{12}^2}, \, 0, \,  \frac{\sigma_3(\rho_{13}-\rho_{12}\rho_{23})}{\sigma_1 \sqrt{1-\rho_{12}^2}}  \right]^\T$;

$
J^{(\{2\}, 3)} = \left[ \frac{\sigma_1(\rho_{13}-\rho_{12}\rho_{23})}{\sigma_3 \sqrt{1-\rho_{23}^2}} , \, 0, \,   \sqrt{1-\rho_{23}^2} \right]^\T$; 
$
J^{(\{2, 3\}, 1)} = \left[ \sqrt{\frac{1-\rho_{12}^2-\rho_{13}^2-\rho_{23}^2+2\rho_{12}\rho_{13}\rho_{23}}{1-\rho_{23}^2}}, \, 0, \, 0  \right]^\T
$.
Finally, we can see that
$ 
J^{(\emptyset, 3)} = \left[ \frac{\rho_{13}\sigma_1}{\sigma_3}, \,  \frac{\rho_{23}\sigma_2}{\sigma_3} , \, 1 \right]^\T      
$;  
$
J^{(\{3\}, 2)} = \left[ \frac{\sigma_1(\rho_{12}-\rho_{13}\rho_{23})}{\sigma_2 \sqrt{1-\rho_{23}^2}}, \,   \sqrt{1-\rho_{23}^2}, \, 0    \right]^\T       
$;
$
J^{(\{3\}, 1)} = \left[\sqrt{1-\rho_{13}^2}, \,   \frac{\sigma_2(\rho_{12}-\rho_{13}\rho_{23})}{\sigma_1 \sqrt{1-\rho_{13}^2}}, \, 0    \right]^\T     
$ 
and    
$
J^{(\{1, 3\}, 2)} = \left[0, \, 0 , \,  \sqrt{\frac{1-\rho_{12}^2-\rho_{13}^2-\rho_{23}^2+2\rho_{12}\rho_{13}\rho_{13}}{1-\rho_{13}^2}}  \right]^\T      
$. 
      
\end{appendices}    
   \bibliographystyle{elsarticle-num}           
                      
                                 
\end{document}